\documentclass[11pt]{article}

\usepackage{a4wide}
\usepackage[utf8]{inputenc}
\usepackage{amsmath}
\usepackage{amssymb}
\usepackage{amsthm}
\usepackage{amsopn}
\usepackage{xcolor}
\usepackage{nicefrac}
\usepackage{enumitem}
\usepackage{cancel}

\newtheorem{theorem}{Theorem}
\newtheorem{lemma}{Lemma}
\newtheorem{remark}{Remark}
\newtheorem{proposition}[theorem]{Proposition}

\usepackage[hidelinks]{hyperref}

\newcommand{\R}{\mathbb{R}}
\newcommand{\RN}{\mathbb{R}^N}

\title{\textsc{Existence issues for a large class of degenerate elliptic equations with nonlinear Hamiltonians}}

\author{I. Birindelli, G. Galise, A. Rodr\'iguez}
\date{}
\vspace{-5ex}

\begin{document}
\maketitle
\begin{abstract}
\noindent We give sufficient conditions for the existence and uniqueness, in bounded uniformly convex domains $\Omega$, of solutions of degenerate elliptic equations depending also on the nonlinear gradient term $H$, in term of the size of $\Omega$, of the forcing term $f$ and of $H$. The results apply to a wide class of equations, having as principal part significant examples, e.g. linear degenerate operators, weighted partial trace operators and the homogeneous Monge-Ampère operator.   
\end{abstract}

\vspace{0.5cm}

{\small
\noindent
\textbf{MSC 2010:} 	35B51, 35D40, 35J25, 35J70, 35J96. 

\smallskip
\noindent
\textbf{Keywords:} Degenerate elliptic equations, viscosity solutions, uniform convex domains.  }
	
	%
	
\section{Introduction}
We study the solvability of the Dirichlet problem
		\begin{equation}\label{genEq}
			\left\{
				\begin{array}{cl}
					F(x,D^2u) + H(Du) = f(x) & \textrm{in }\Omega\\
					u = 0 & \textrm{on }\partial\Omega
				\end{array}
			\right.
		\end{equation}
under mild assumptions on the degenerate elliptic operator $F$, in bounded uniformly 
convex domains $\Omega\subset\R^N$ and with power type  Hamiltonian $H$.  
The relevance of \eqref{genEq}, at least for $F(x,D^2u)=\Delta u$, is well known and we will not attempt to list the papers devoted to that problem.  We just wish to mention that, beside its intrinsic relevance it is important to study existence of  solutions of \eqref{genEq} because it is related to the so called ergodic constant  as it is well described in the 
the work of Porretta \cite{PorLin}, see also Remark \ref{por}.

\smallskip We will state the precise assumptions in the body of the paper but the prototype equation to be kept in mind is
\begin{equation}\label{mainexample}
	a(x)\lambda_N(D^2u)+b|Du|^p=f(x)\quad\text{in $\Omega$},
\end{equation}
where $b,p>0$ and $a(x)\geq \beta>0$, $f(x)$ are continuous functions. Here, and in the whole paper, 
for any $N\times N$ symmetric matrix $X$, 
 $\lambda_1(X),\dots,\lambda_N(X)$ are the ordered increasing eigenvalues. 

\smallskip
The scope is to give conditions on $f$, $\Omega$ and $H$ in order to prove existence 
and uniqueness of viscosity solutions of \eqref{genEq}.

We will describe some of these conditions in the prototype equation \eqref{mainexample}. When $p>1$, 
we prove existence and uniqueness for uniformly convex domains $\Omega$  such that
		\begin{equation}\label{suffcondint}
		\Omega=\bigcap_{y\in Y}B_R(y)\qquad\text{for some}\qquad R\leq{\bar R}:=\frac{\beta(p-1)^{\frac{p-1}{p}}}{p\,b^\frac1p\,\left\|{f}^{^-}\right\|_\infty^{\frac{p-1}{p}}}\,,
		\end{equation}
		where, as usual, $B_R(y)$ is the ball of radius $R$ with center at $y$ and $Y\subset\R^N$. 
Observe that in the above condition,  $R$ measures the convexity of $\Omega$. Indeed if the domain is $C^2$, which may not to be case, then $R\geq\frac{1}{\kappa}$ where $\kappa=\min\left\{\kappa_i(x):\,i=1,\ldots,N-1,\;x\in\partial\Omega\right\}$ and $\kappa_1(x),\ldots,\kappa_{N-1}(x)$ are the principal curvatures of $\partial\Omega$ at $x$, see \cite{BGI}.

The equation  \eqref{mainexample} can be equivalently written as
	$$
	a(x)\max_{|\xi|=1}\left\langle D^2u\,\xi,\xi\right\rangle+b|Du|^p=f(x)\quad\text{in $\Omega$.}
	$$
Note that the left hand side of the above equation is, roughly speaking, a big quantity in view of the \lq\lq maximum\rq\rq\,  and of the  positive gradient term. Then a negative right hand side forces the largest eigenvalue $\lambda_N(D^2u)$   to be very negative to balance $f(x)$ and so a solution $u$ of \eqref{mainexample}  to be concave in $\Omega$. This consideration suggests that the convexity of the domain is somehow needed to allow the principal part to absorb both the gradient part of the equation and the data $f$.

The existence of solution is done via a Perron's method. We wish to emphasize that the construction of a supersolution null on the boundary is quite original and far from obvious. It relies on the existence of $C^2$-radial solutions of
	\begin{equation}\label{extremeBallEq2}
			\left\{
				\begin{array}{cl}
					\beta\lambda_N(D^2u) + b|Du|^p = -M & \textrm{in } B_R\\
					u = 0 & \textrm{on } \partial B_R,
				\end{array}
			\right.
		\end{equation}
with $b,M>0$, which exist as soon as $R\leq\bar R$, see  Proposition \ref{existOnBallProp}. In a similar fashion, we also construct singular radial solutions in $B_R(0)\backslash\left\{0\right\}$, see Proposition \ref{blow-up solutions}.

Let us remark that \eqref{suffcondint}, which gives a sufficient condition for the solvability of the Dirichlet problem (1), is in fact not very far from being sharp. 
Indeed we will infer by Remark \ref{blowup} and Proposition \ref{CP} (which gives uniqueness of solutions) that  \eqref{extremeBallEq2} cannot have $C^2$-radial solutions if ${R>\bar R}$. 

The necessity of balancing the size of the domain, the forcing term $f$ and the coefficient of the gradient part is already present in the uniformly elliptic case, both linear and nonlinear, see e.g. \cite{BPT} and  \cite{GMP,S}.
		
In the sublinear case, i.e. for $p\in (0,1)$, the size of the uniform convexity of 
$\Omega$ doesn't play a role, but we need to restrict in a qualitative way the class of 
functions $f$ in order to ensure that the comparison principle holds.

The existence and nonexistence issues in the case $p=1$ have been already treated in \cite{BGI,BGI2}.

	
\smallskip We wish to emphasize that one of the main interest of this paper is that the 
class of equations we consider is very large, since the conditions on $F$ and $H$ are very mild. In particular $F$ needs not be neither in divergence form, nor linear and nor uniformly elliptic. 
We will show in Section 2 that, e.g., the following operators are included:

\begin{enumerate}
\item Linear degenerate elliptic operators $$F(x,D^2u)=\text{Tr}(\Sigma^T(x)\Sigma(x)D^2u)$$ as long as $\lambda_N(\Sigma^T(x)\Sigma(x))>0$, e.g. $F(x,D^2u)=u_{x_1x_1}$;
\item Nonlinear degenerate operators, functions of the eigenvalues, e.g. $$F(D^2u)=\sum_{i=1}^N\alpha_i\lambda_i(D^2u)$$ with $\alpha_i\geq 0$ and $|(\alpha_1,\cdots,\alpha_N)|>0$;
\item The homogenous Monge-Amp\`ere operator $$F(D^2u)=\left(\det(D^2u)\right)^\frac1N, \qquad D^2u\geq0.$$
We remark that, since the map $u\mapsto \det(D^2u)$ is elliptic only if $u$ is constrained to be in the positive cone of convex functions, we shall treat this case separately at the end of Sections \ref{super}-\ref{sub};
\end{enumerate}
Concerning the Hamiltonian $H$, let us mention that we don't require that the Hamiltonian be convex or be precisely a power. For example, compact perturbations of power type Hamiltonians, i.e.
$$H(Du)=|Du|^p+\phi(Du)$$
with $\phi\in C^1_c(\R^N)$, satisfy the assumptions.

	\section{Notations and basic assumptions}
	Let $\Omega\subset \RN$ be an open, bounded domain. We say that $\Omega$ is uniformly convex if there exist $R>0$ and $Y\subset\RN$, depending on $\Omega$, such that 
	\begin{equation}\label{unifconv}
	\Omega=\bigcap_{y\in Y}B_R(y).
	\end{equation}
	
	Let $\mathbb S^N$ be the linear space of  $N\times N$ symmetric matrices. For any $X\in\mathbb S^N$ we denote by $\lambda_i(X)$, for $i=1,\ldots,N$, the eigenvalues of $X$ arranged in nondecreasing order:
	$$
	\lambda_1(X)\leq\ldots\leq\lambda_N(X).
	$$
	
The norm of $X\in\mathbb S^N$ is $\left\|X\right\|=\max\left\{|\lambda_1(X)|,|\lambda_N(X)|\right\}$.	
	
	\bigskip
	
	\noindent We introduce the hypotheses on $F\in C\left(\Omega\times\mathbb S^N;\R\right)$ and $H\in C(\RN;\R)$:

	\begin{enumerate}[label=(F\arabic*)]
		\item\label{ellipticity} there exists  $\beta>0$ such that for any  $(x,X)\in\Omega\times\mathbb S^N$
					\begin{equation}\label{deg}
						F(x,X+Y) - F(x,X) \leq\beta\lambda_N(Y)\qquad\forall\, Y\leq0\,;
					\end{equation}
		\item\label{hom} $F$ is positive $1$-homogeneous, i.e. $F(x,\sigma X) = \sigma F(x,X)$ for all $\sigma>0$, $(x,X)\in\Omega\times\mathbb S^N$.
		
		\end{enumerate}
		As far as $H$ is concerned we shall assume either:
		\begin{enumerate}[label=(H\arabic*)]
		\item\label{H1}
					there exist $b,c>0$,  $p>1$ such that for any $\xi,\eta\in\RN$
					\begin{equation}\label{gradGrowth}
						H(\sigma\eta+(1-\sigma)\xi)- \sigma H(\eta)\leq (1-\sigma)\left(b|\xi|^p+c\right)\qquad\forall \sigma\in[0,1]
					\end{equation}
					and moreover $H$ is bounded from below, i.e. $H(\xi)\geq-d$ for some $d>0$; 
					
					\medskip
					or
					\medskip
					\item\label{H2}  there exist $b,c>0$, $p\in(0,1)$ such that 
	\begin{equation*}
	\begin{split}
	\varepsilon H(\xi)&\leq H(\varepsilon\xi)\qquad\forall(\varepsilon,\xi)\in(0,1)\times\RN\\
	0\leq H(\xi)&\leq b|\xi|^p+c\qquad\forall\xi\in\RN.
	\end{split}
	\end{equation*}
					
		\end{enumerate}
					%
		
		\noindent Finally we shall require that the comparison principle assertion holds for \eqref{genEq} with strict inequality, i.e. 
		
		\begin{enumerate}
			\item[(CC)] if $u\in USC(\overline\Omega)$ and $v\in LSC(\overline\Omega)$ are respectively subsolution and strict supersolution of \eqref{genEq}, then $u\leq v$ in $\overline\Omega$.
		\end{enumerate}
		
	\noindent
	A well known sufficient condition for the validity (CC) is: 
	
	\begin{itemize}
		\item[(SuffCC)]\label{Ishii} there exists   a modulus of continuity $\omega$ such that 
		$$F(x,X)-F(y,Y)\leq\omega(\alpha|x-y|^2+|x-y|)$$
		whenever $(\alpha,x,y,X,Y)\in\mathbb R_+\times\Omega^2\times(\mathbb S^N)^2$ and 
		\begin{equation}\label{condcomp}
		-3\alpha\left(
		\begin{array}{cc}
		I & 0\\0 & I
		\end{array}
		\right)
		\leq
		\left(
		\begin{array}{cc}
		X & 0\\0 & -Y
		\end{array}
		\right)
		\leq3\alpha\left(
		\begin{array}{cc}
		I & -I\\-I & I
		\end{array}
		\right).
		\end{equation}
	\end{itemize}
	
	Let us recall that (SuffCC) is always fulfilled if $F$ is a degenerate elliptic operator independent of the $x$-variable, see \cite[Example 3.6]{CIL}
	
	\begin{remark}
	\rm It is possible to consider Hamiltonian $H$ depending  on $x$ and satisfying \ref{H1}-\ref{H2} as long as the comparison principle (CC) still holds. 
	\end{remark}
	
	\subsection{Comments and examples}
	\subsubsection{The principal part $F$}\label{examples F}
	The assumption \ref{ellipticity} can be equivalently stated as follows: for any $(x,X)\in\Omega\times\mathbb S^N$
	\begin{equation}\label{deg2}
						F(x,X+Y) - F(x,X) \geq\beta\lambda_1(Y)\qquad\forall\, Y\geq0\,.
					\end{equation}
	It readily follows from \eqref{deg2} that  $$F(x,X+Y)\geq F(x,X)\qquad\forall\, Y\geq0,$$ i.e. $F$ is degenerate elliptic.
	
	The class of operators satisfying the ellipticity condition   \eqref{deg2}, as well as \ref{hom}-(CC), is quite large. It includes some important examples:
	\noindent
	\begin{enumerate}
		\item \textbf{Strictly elliptic operators}
		
		Let $F:\Omega\times\mathbb S^N\mapsto\R$ be continuous and strictly elliptic, i.e. there exists $\nu>0$ such that for any $(x,X)\in\Omega\times\mathbb S^N$
		$$
		F(x,X+Y)-F(x,X)\geq\nu\,\text{Tr}(Y)\qquad\forall\, Y\geq0.
		$$
		Then $F$ satisfies \eqref{deg2} with $\beta=N\nu$.\\ 
		Concerning the validity of (CC), some sufficient conditions can be found in \cite[Theorem III.1]{IL}. Here we just recall the following one: 
		$$
		\left|F(x,X)-F(y,X)\right|\leq\overline\omega(|x-y|(1+\left\|X\right\|))\qquad\forall(x,y,X)\in\Omega^2\times\mathbb S^N,
		$$
		where $\overline\omega$ is a modulus of continuity such that $\frac{\overline\omega(r)}{1+r}$ is bounded for $r\geq0$. 
		
		\item \textbf{Degenerate linear operators}
		
		For $(x,X)\in\Omega\times\mathbb S^N$ we let $$F(x,X)=\text{Tr}(\Sigma^T(x)\Sigma(x)X),$$ where $\Sigma^T(x)\Sigma(x)\geq0$ and  
		\begin{equation}\label{131eq1}
		\inf_{x\in\Omega}\lambda_N(\Sigma^T(x)\Sigma(x))>0.
		\end{equation}
		To check \eqref{deg2}, let $\left\{v_1(x),\ldots,v_N(x)\right\}$ be an orthonormal basis of eigenvectors of the matrix $\Sigma^T(x)\Sigma(x)$. Then		
		\begin{equation*}
		\begin{split}
		F(x,X+Y)-F(x,X)&=\text{Tr}(\Sigma^T(x)\Sigma(x)Y)\\&=\sum_{i=1}^N\left\langle Yv_i(x),\Sigma^T(x)\Sigma(x)v_i(x)\right\rangle\\&\geq\left(\sum_{i=1}^N\lambda_i(\Sigma^T(x)\Sigma(x))\right)\lambda_1(Y).
		\end{split}
		\end{equation*}
		The structural condition \eqref{deg2} is then satisfied with $\displaystyle\beta=\inf_{x\in\Omega}\lambda_N(\Sigma^T(x)\Sigma(x))$. The operator $F$ also satisfies (SuffCC) if $\Sigma(x)$ is assumed to be Lipschitz continuous, see \cite[Example 3.6]{CIL}.
		
		\smallskip
		In particular, choosing $\Sigma(x)=e_i\otimes e_i$ for  $i=1,\ldots,N$, where $\left\{e_1,\ldots,e_N\right\}$ is the standard basis of $\R^N$, we infer that the  equations  $$\frac{\partial^2u}{\partial x_i^2}+H(Du)=f(x)$$
		fit in our framework.
		\item \textbf{Nonlinear degenerate operators, functions of the eigenvalues}
		
		Let $\alpha=(\alpha_1,\ldots,\alpha_N)$ be such that $\alpha_i\geq0$ for $i=1,\ldots,N$ and $|\alpha|>0$. Set
		\begin{equation}\label{exdeg}
		F(X)=\sum_{i=1}^N\alpha_i\lambda_i(X).
		\end{equation}
		It is easy to check that \eqref{deg2} is satisfied with $\displaystyle\beta=\sum_{i=1}^N\alpha_i$, just using the inequality
		\begin{equation}\label{eigenvalueinequality}
		\lambda_i(X+Y)-\lambda_i(X)\geq\lambda_1(Y)
		\end{equation}
		which holds for any $i\in\left\{1,\ldots,N\right\}$ and for any $X,Y\in\mathbb S^N$.
		 Moreover, since $F$ in \eqref{exdeg} is independent of the $x$-variable, then (SuffCC) holds as well.		
		
		By \eqref{exdeg} we include in particular the truncated Laplacians (see \cite{BGI,BGI2}) 
		$${\mathcal P}_k^+(X)=\lambda_1(X)+\ldots+\lambda_k(X)\quad\,\text{and}\quad\, {\mathcal P}_k^+(X)=\lambda_{N-k+1}(X)+\ldots+\lambda_N(X),$$ 
the min-max operators considered in  \cite{FV} 
$$F(X)=\lambda_1(X)+\lambda_N(X)$$
		and $$F(X)=\lambda_i(X)$$
		for some $i\in\left\{1,\ldots,N\right\}$, see \cite{BR}.
		
		\smallskip
		
		We stress that the ellipticity condition \eqref{deg2} is required to be satisfied only for nonnegative matrices $Y$ and not for any $Y\in\mathbb S^N$. This fact allows us to consider, for instance, the operator
		$$
		F(X)=\lambda_i(X)-\left(\lambda_j(X)\right)^-
		$$
		where $i,j\in\left\{1,\ldots,N\right\}$ and $t^-=\max\left\{-t,0\right\}$ for any $t\in\R$. It is easy to check that $F(X)$ is  positive homogeneous of degree one and it satisfies \eqref{deg2} for any $\beta\in(0,1]$. For this is sufficient to use the inequality \eqref{eigenvalueinequality} and the monotonicity of the map $t\mapsto t^-$. Moreover for $X=0$ and $Y=I$  it turns out that
		$$
		F(X+Y)-F(X)=1=\lambda_1(Y),
		$$
	showing that the best constant $\beta$ we can take in \eqref{deg2} is $\beta=1$. On the other hand, if \eqref{deg2} were valid  also for $Y\leq0$, then for $X=0$ and $Y=-I$ we would have
	$$
	F(X+Y)-F(X)=-2=2\lambda_1(Y).
	$$
	Hence $\beta\geq2$ which is in contradiction to $\beta\leq1$.
				
		\medskip
		Let's go back now to the  prototype equation \eqref{mainexample}. Suppose that $a(x)\in\text{Lip}(\Omega)$, with constant $L$, and $\displaystyle \inf_{x\in\Omega}a(x)>0$. Let 
\begin{equation*}
F(x,X)=a(x)\lambda_N(X).
\end{equation*}
For any $Y\geq0$ we have
$$
a(x)\lambda_N(X+Y)-a(x)\lambda_N(X)\geq \left(\inf_{x\in\Omega}a(x)\right)\lambda_1(Y).
$$
Then condition \eqref{deg2} is satisfied with $\beta=\displaystyle \inf_{x\in\Omega}a(x)$. \\ We claim that 
(SuffCC) holds.  From \eqref{condcomp}, for any $\xi,\eta\in\mathbb R^N$ we have
$$
\left(
		\begin{array}{cc}
		X & 0\\0 & -Y
		\end{array}
		\right)
		\left(
		\begin{array}{c}
		\sqrt{a(x)}\,\xi \\ \sqrt{a(y)}\,\eta
		\end{array}
		\right)\cdot\left(
		\begin{array}{c}
		\sqrt{a(x)}\,\xi \\ \sqrt{a(y)}\,\eta
		\end{array}
		\right)
		\leq3\alpha\left(
		\begin{array}{cc}
		I & -I\\-I & I
		\end{array}
		\right)\left(
		\begin{array}{c}
		\sqrt{a(x)}\,\xi \\ \sqrt{a(y)}\,\eta
		\end{array}
		\right)\cdot\left(
		\begin{array}{c}
		\sqrt{a(x)}\,\xi \\ \sqrt{a(y)}\,\eta
		\end{array}
		\right)
$$
which leads to 
\begin{equation*}
a(x)X\xi\cdot\xi-a(y)Y\eta\cdot\eta\leq3\alpha\left(a(x)-2\sqrt{a(x)a(y)}\,\xi\cdot\eta+a(y)\right).
\end{equation*}
In particular, choosing $\xi\in\RN$ such that $X\xi\cdot\xi=\lambda_N(X)$, we have
\begin{equation}\label{17eq1}
F(x,X)-a(y)Y\eta\cdot\eta\leq3\alpha\left(a(x)-2\sqrt{a(x)a(y)}\,\xi\cdot\eta+a(y)\right).
\end{equation}
Minimizing both sides of \eqref{17eq1}, among  all $\eta\in\mathbb R^N$ such that $|\eta|=1$, we obtain
\begin{equation*}
\begin{split}
F(x,X)-F(y,Y)&\leq3\alpha\left(a(x)-2\sqrt{a(x)a(y)}+a(y)\right)\\&=3\alpha\left(\sqrt{a(x)}-\sqrt{a(y)}\right)^2.
\end{split}
\end{equation*}
Then using the Lipschitz continuity of $a(x)$
$$
F(x,X)-F(y,Y)\leq\frac{3L^2}{4\displaystyle\inf_{x\in\Omega}a(x)}\,\alpha|x-y|^2.
$$
Hence (SuffCC) is satisfied with $$\omega(r)=\frac{3L^2}{4\displaystyle\inf_{x\in\Omega}a(x)}\,r.$$

\item \textbf{Monge-Ampère operator}

Let
\begin{equation}\label{MA}
F(X)=\left(\det(X)\right)^\frac1N,\qquad X\geq0,
\end{equation}
	be the 1-homogeneous Monge-Ampère operator. The restriction $X\geq0$ is natural in order to  ensure the ellipticity of the map $X\mapsto \det(X)$. In addition the operator \eqref{MA}  satisfies, within the class of nonnegative symmetric matrices,  the structural condition \eqref{deg2}. Indeed, using the matrix identity (see \cite[V.3]{IL}) 
		\begin{equation*}
		\left(\det(X)\right)^\frac1N=\inf\left\{\text{Tr}(XB)\,:\;B\geq0,\,\det B=\frac{1}{N^N}\right\},
		\end{equation*}
		 we then obtain
		\begin{equation*}
		\begin{split}
		\left(\det(X+Y)\right)^\frac1N-\left(\det(X)\right)^\frac1N&\geq\inf\left\{\text{Tr}(YB)\,:\;B\geq0,\,\det B=\frac{1}{N^N}\right\}\\
		&=\left(\det(Y)\right)^\frac1N\geq\lambda_1(Y)\qquad\forall X,Y\geq0,
		\end{split}
		\end{equation*}
		that is  \eqref{deg2} holds with $\beta=1$.

		\item \textbf{Bellman–Isaacs type operators}
		
		We consider a two-parameters family of 1-homogeneous operators $\left\{F_{a,b}\right\}$ depending on $a$ and $b$ running in some sets of indexes $\mathcal A$ and $\mathcal B$. Let 
		\begin{equation}\label{bellman}
		F(x,X)=\sup_a\inf_b F_{a,b}(x,X).
		\end{equation}If we assume that $F_{a,b}$ satisfies  \eqref{deg2} with some $\beta>0$,  independent on  $a\in\mathcal A$ and $b\in\mathcal B$, as well as (SuffCC) with a common modulus of continuity $\omega$, then \eqref{deg2}-(suffCC) are in turn satisfied by \eqref{bellman} with the same $\beta$ and $\omega$.

	\end{enumerate}

	
	\subsubsection{The first order term $H$}
	Typical examples we have in mind are $$H(\xi)=b|\xi|^p\quad\text{and}\quad H(\xi)=\left\langle A\xi,\xi\right\rangle^{\frac p2}$$ with $0\leq A\leq b^{\frac2p}I$ in $\mathbb S^N$. They satisfies (H1) and (H2) for $p>1$ and $p\in(0,1)$ respectively.
	
	\smallskip
	A comment on \eqref{gradGrowth} in (H1) is in order. Such condition  is a convexity type assumption which is in particular satisfied, with $c=0$, by convex Hamiltonian such that $H(\xi)\leq b|\xi|^p$.   But we point out that \ref{H1} also includes nonconvex Hamiltonians, as shown in the next Lemma.

	\begin{lemma}
	If $\phi\in C^1_c(\RN)$ then $H(\xi)=|\xi|^p+\phi(\xi)$ satisfies \ref{H1}.
	\end{lemma}
	\begin{proof}It is clear that $H$ is bounded from below. We claim that \eqref{gradGrowth} holds with $b=1$ and $c$ large enough (depending on $\phi$). \\ Let $\eta,\xi\in\RN$, $\sigma\in[0,1]$ and let $[\xi,\eta]=\left\{\sigma\eta+(1-\sigma)\xi\,:\;\sigma\in[0,1]\right\}$ be  the segment joining $\xi$ and $\eta$.
If $[\xi,\eta]\cap\text{supp}(\phi)=\emptyset$, then just using the convexity of the map $\xi\mapsto|\xi|^p$, we infer that \eqref{gradGrowth} is satisfied for any  $c\geq0$. So  we shall assume from now on that $[\xi,\eta]\cap\text{supp}(\phi)\neq\emptyset$. Fix $R$ such that $\text{supp}(\phi)\subset B_R$ and 
\begin{equation}\label{2912eq1}
R^p=\max_{\xi\in\overline B_R} H(\xi).
\end{equation}
Set \begin{equation}\label{c}
c=\max\left\{R^p,2R\left\|D\phi\right\|_\infty+\left\|\phi\right\|_\infty\right\}.
\end{equation}

\smallskip
\noindent\textbf{Case 1: $[\xi,\eta]\subset \overline B_R$.}\\
By the convexity of the map $\xi\mapsto|\xi|^p$ and using \eqref{c} we have
\begin{equation*}
\begin{split}
H(\sigma\eta&+(1-\sigma)\xi)-\sigma H(\eta)\\
&=|\sigma\eta+(1-\sigma)\xi|^p-\sigma|\eta|^p+\phi(\sigma\eta+(1-\sigma)\xi)-\phi(\eta)+(1-\sigma)\phi(\eta)\\
&\leq(1-\sigma)|\xi|^p+\left\|D\phi\right\|_\infty|(1-\sigma)(\xi-\eta)|+(1-\sigma)\phi(\eta)\\
&\leq(1-\sigma)\left(|\xi|^p+2R\left\|D\phi\right\|_\infty+\left\|\phi\right\|_\infty\right)\\
&\leq(1-\sigma)\left(|\xi|^p+c\right).
\end{split}
\end{equation*}

\smallskip
\noindent\textbf{Case 2:  $\eta\notin \overline B_R$ and $\xi\in \overline B_R$ (or $\eta\in \overline B_R$ and $\xi\notin \overline B_R$).}\\
We assume that $\eta\notin \overline B_R$ and $\xi\in \overline B_R$, the other case being similar.  We first observe that for any $\sigma\in[0,1]$ such that 
$$
\sigma\eta+(1-\sigma)\xi\in \overline B_R,
$$
then 
$$
H(\sigma\eta+(1-\sigma)\xi)\leq\max_{\overline B_R} H,
$$
while, since $\eta\notin \overline B_R$, it holds that  
\begin{equation*}
\sigma H(\eta)+(1-\sigma)(|\xi|^p+c)\geq\sigma R^p+(1-\sigma)c\geq R^p= \max_{\overline B_R} H.
\end{equation*}
If instead 
$$
\sigma\eta+(1-\sigma)\xi\notin \overline B_R,
$$
then
\begin{equation*}
\begin{split}
H(\sigma\eta+(1-\sigma)\xi)-\sigma H(\eta)&=|\sigma\eta+(1-\sigma)\xi|^p-\sigma|\eta|^p\\
&\leq(1-\sigma)|\xi|^p\leq(1-\sigma)(|\xi|^p+c).
\end{split}
\end{equation*}

\smallskip
\noindent\textbf{Case 3:  $\eta,\xi\notin \overline B_R$.}\\
Let 
\begin{equation*}
\overline H(\xi)=\max\left\{|\xi|^p,R^p\right\}.
\end{equation*}
Note that $\overline H$ is convex in $\RN$ since it is maximum of convex functions. Moreover, by \eqref{2912eq1}, $\overline H(\xi)\geq H(\xi)$ for any $\xi\in\RN$. Then, since $\eta,\xi\notin \overline B_R$, we obtain
\begin{equation*}
\begin{split}
H(\sigma\eta+(1-\sigma)\xi)&\leq\overline H(\sigma\eta+(1-\sigma)\xi)\\
&\leq\sigma\overline H(\eta)+(1-\sigma)\overline H(\xi)\\
&=\sigma H(\eta)+(1-\sigma)|\xi|^p\leq \sigma H(\eta)+(1-\sigma)(|\xi|^p+c)
\end{split}
\end{equation*}
	as we wanted to show.
	\end{proof}
	
	\subsection{Maximal and minimal inequalities}
The structural conditions \ref{ellipticity}-(H1) identify two extremal inequalities which act as  barriers for the whole class of equations we are dealing with.\\	
First observe that, taking $\sigma=0$ in \eqref{gradGrowth}, then for any $\xi\in\RN$
	$$
	-d\leq H(\xi)\leq b|\xi|^p+c.
	$$
From this and using \eqref{deg}, \eqref{deg2} and $F(x,0)=0$ $\forall x\in\Omega$, we infer that the inequalities
	\begin{eqnarray}
	\label{maxinequ}\beta\lambda_N(D^2u)+b|Du|^p\leq f(x)-c &\text{in $\Omega$}\\
	\label{maxinequ2}\beta\lambda_1(D^2u)\geq f(x)+d&\text{in $\Omega$}
	\end{eqnarray}
		play the role of  maximal and minimal inequalities  within the class of concave and convex functions respectively. Namely every concave supersolution of \eqref{maxinequ} is   supersolution of \eqref{genEq} and any convex subsolution of \eqref{maxinequ2} is in turn  subsolution of \eqref{genEq}.
		
		\medskip
		
	\noindent	Similarly, under the sublinear assumption (H2), the inequalities  
		\begin{eqnarray}
	\label{submaxinequ}\beta\lambda_N(D^2u)+b|Du|^p\leq f(x)-c &\text{in $\Omega$}\\
	\label{submaxinequ2}\beta\lambda_1(D^2u)\geq f(x)&\text{in $\Omega$}
	\end{eqnarray}
	are extremal for concave and convex functions respectively.

\section{Superlinear Hamiltonians}\label{super}
Throughout this section we shall assume $p>1$.
\subsection{Existence result on balls}
 Let $M>0$. We consider   		
\begin{equation}\label{extremeBallEq}
			\left\{
				\begin{array}{cl}
					\beta\lambda_N(D^2u) + b|Du|^p = -M & \textrm{in } B_R(y)\\
					u = 0 & \textrm{on } \partial B_R(y).
				\end{array}
			\right.
		\end{equation}

	\begin{proposition}\label{existOnBallProp}
	Suppose that
	\begin{equation}\label{condR}
	R\leq{\bar R}:=\frac{\beta(p-1)^{\frac{p-1}{p}}}{p\,b^\frac1p\,M^{\frac{p-1}{p}}}\,.
	\end{equation}
		 Then there exists  $u$ radial solution  of \eqref{extremeBallEq} such that $$u\in C^2(\overline{B_R(y)})\quad\text{ if}\;\;R<{\bar R}\,,\qquad  u\in C^2(B_R(y))\cap C^1(\overline{B_R(y)})\quad\text{if}\;\; R=\bar R.$$
		Moreover
		\begin{equation}\label{unifestimate}
		\left\|u\right\|_{C^1(\overline{B_R(y)})}\leq \left(\frac{\beta}{{\bar R}pb}\right)^{\frac{1}{p-1}}({\bar R}+1).
		\end{equation}
	\end{proposition}
	
	\begin{remark}{\rm
		As $p\to1^+$, the condition \eqref{condR} reduces to the one given in \cite[Proposition 17]{BGI2}.}
	\end{remark}
	
		\begin{proof}
		We first consider the case ${R=\bar R}$.	Slightly abusing notation, we write $u(x) = u(r)$ with $r=|x-y|$. Our candidate for radial solution of \eqref{extremeBallEq} is the solution of the ODE
				\begin{equation}\label{radialExtremeBallEq}
					\left\{
						\begin{array}{cl}
							\beta\frac{u'(r)}{r} + b(-u'(r))^p = -M, & r\in (0,\bar R)\\
							u(\bar R) = 0.
						\end{array}
					\right.
				\end{equation}
			In fact, a solution of \eqref{radialExtremeBallEq} corresponds to a solution of \eqref{extremeBallEq}  if 
				\begin{equation}\label{radialU'Zero}
					\lim_{r\to0^+}u'(r)=0,
				\end{equation}
			and, for all $r\in (0,\bar R)$,
				\begin{equation}\label{radialUDecreas}
					u'(r) < 0
				\end{equation}
			and
				\begin{equation}\label{radialEigenIneq}
					\frac{u'(r)}{r}\geq u''(r).
				\end{equation}
Observe that, for any fixed  $r\in (0,\bar R)$, the equation 
\begin{equation}\label{eq1}
\beta\frac{u'(r)}{r} + b(-u'(r))^p = -M 
\end{equation}
is in a sense an ``algebraic'' equation, given that it involves only the first derivatives of the unknown $u$. 
Therefore let us consider the function
$$
\varphi(r,s)=-\beta\frac sr+bs^p+M\qquad \text{for $r,s>0$}
$$
and note that 
\begin{equation}\label{eq4}
\text{$u(r)$ is solution of \eqref{eq1} if, and only if,\, $\varphi(r,-u'(r))=0$.}
\end{equation}
				%
We claim that  there exists a nonnegative function $s_0=s_0(r)\in C^1([0,\bar R))\cap C([0,\bar R])$ such that 
\begin{enumerate}
  \item[(i)] $\varphi(r,s_0(r))=0$ for $r\in (0,\bar R)$;
	\item[(ii)] $s_0(r)$ is increasing;
	\item[(iii)]  $s_0(0)=0$. 
	
\end{enumerate}
Fix $r\in(0,\bar R)$ and let $s_1(r)=\left(\frac{\beta}{rpb}\right)^{\frac{1}{p-1}}$. Note that $s\mapsto\varphi(r,s)$ is a smooth strictly convex function in $[0,\infty)$, which is decreasing for $s\in\left[0,s_1(r)\right]$ and then increasing for $s\geq s_1(r)$. Moreover $\varphi(r,0)=M>0$ and $\displaystyle\lim_{s\to+\infty}\varphi(r,s)=+\infty$. Exploiting the assumption \eqref{condR}, it is easy to check that
$$
\varphi\left(r,s_1(r)\right)<0.
$$
Hence there exists $s_0=s_0(r)\in(0,s_1(r))$, the \lq\lq first zero\rq\rq\ of $\varphi(r,s)$, such that
$$
\varphi(r,s_0(r))=0\quad\text{and}\quad\varphi(r,s)>0\;\;\;\forall s\in[0,s_0(r)). 
$$
Moreover
\begin{equation}\label{eq2}
\partial_s\varphi(r,s_0(r))=-\frac\beta r+bp(s_0(r))^{p-1}<-\frac\beta r+bp(s_1(r))^{p-1}=0
\end{equation}
 and, by the implicit function theorem, we have that $r\mapsto s_0(r)\in C^1\left((0,\bar R)\right)$. In addition, by \eqref{eq2},
\begin{equation}\label{5eq1}
s'_0(r)=-\frac{\partial_r\varphi(r,s_0(r))}{\partial_s\varphi(r,s_0(r))}=-\frac{\beta s_0(r)}{r^2\partial_s\varphi(r,s_0(r))}>0.
\end{equation}
This shows (i)-(ii).  In order to prove (iii) fix any $\alpha\in(0,1)$. For $r$ positive and small enough we have
\begin{equation*}
\varphi(r,r^\alpha)=-\frac{\beta}{r^{1-\alpha}}+br^{\alpha p}+M<0=\varphi(r,s_0(r)).
\end{equation*}
This implies that $s_0(r)<r^\alpha$, then $s_0(0):=\displaystyle\lim_{r\to0^+}s_0(r)=0$. Let us show that $s_0(r)$ can be extended continuously at $r=\bar R$. In this case $s_0(\bar R)=\left(\frac{\beta}{\bar Rpb}\right)^{\frac{1}{p-1}}$. Note that 
\begin{equation}\label{5eq2}
\varphi(\bar R,s_0(\bar R))=\partial_s\varphi(\bar R,s_0(\bar R))=0,
\end{equation}
so we cannot directly apply the implicit function theorem. Nevertheless, by (ii) there exists $L=\displaystyle \lim_{r\to \bar R^{^-}}s_0(r)$. Taking into account that $s_0(r)< s_1(r)$ for any $r\in(0,\bar R)$, then
$$L\leq\lim_{r\to R^{^-}}s_1(r)=s_0(\bar R)$$
and moreover
$$
0=\lim_{r\to \bar R^{^-}}\varphi(r,s_0(r))=\varphi(\bar R,L).
$$ 
This implies that  $L=s_0(\bar R)$, since by the definition of $s_0(\bar R)$ we have that $\varphi(\bar R,s)>0$ for any $s\in[0,s_0(\bar R))$.\\
It remains to prove that $s'_0(r)$ can be extended continuously at $r=0$. By \eqref{5eq1}
it holds that
$$
\lim_{r\to0^+}s'_0(r)=\lim_{r\to0^+}\frac{M+b(s_0(r))^p}{\beta-bpr(s_0(r))^{p-1}}=\frac{M}{\beta}.
$$ 
This show that $s_0\in C^1\left([0,\bar R)\right)$. The proof of the claim is then complete.

\medskip
\noindent			
For $r\in[0,\bar R]$ let
\begin{equation}\label{u0}
u(r):=\int_r^{\bar R}s_0(t)\,dt.
\end{equation}
The function $u=u(r)$ is well defined since $s_0\in C([0,\bar R])$. Let us prove that it  is in fact a radial solution of \eqref{extremeBallEq}. We proceed to check that  $u$  satisfies  \eqref{radialExtremeBallEq}-\eqref{radialEigenIneq}. 
It is trivial that $u(\bar R)=0$. Moreover, since $-u'(r)=s_0(r)$, by \eqref{eq4} and (i) we infer that $u$ is solution of \eqref{radialExtremeBallEq}. Using (ii)-(iii), then conditions  \eqref{radialUDecreas}-\eqref{radialU'Zero} are respectively satisfied.
To prove \eqref{radialEigenIneq}, let us observe that  by \eqref{radialExtremeBallEq},
\begin{equation}\label{eq5}
					\frac{u'(r)}{r} = -\frac1\beta\left(M+ b(-u'(r))^p\right)
				\end{equation}
and since $u'(r)$ is monotone decreasing  by (ii), then
 $\frac{u'}{r}$ is decreasing as well (since the right hand side of \eqref{eq5} is decreasing). Thus, for $r\in(0,\bar R)$,
				\begin{equation*}
					\left(\frac{u'(r)}{r}\right)' = \frac{u''(r)}{r} - \frac{u'(r)}{r^2} \leq 0
				\end{equation*}
			and therefore $u''(r) - \frac{u'(r)}{r}\leq 0$, which is precisely \eqref{radialEigenIneq}. 	
		
Using the monotonicity of $s_0(r)$ and the definition of $u$, we obtain 
\begin{equation}\label{5eq3}		
\begin{split}
 \left\|u\right\|_\infty&=u(0)\leq \bar Rs_0(\bar R)\\
\left\|u'\right\|_\infty&=-u'(\bar R)= s_0(\bar R)
\end{split}
\end{equation}		
which leads to \eqref{unifestimate} in the case $R=\bar R$.

The case $R<\bar R$ easily follows from the previous one. Indeed, if we denote by $u_{{\bar R}}$ the radial  solution we found in the case $R={\bar R}$, then it is clear that $u(r):=u_{{\bar R}}(r)-u_{{\bar R}}(R)$ is a radial solution of  \eqref{extremeBallEq}. Moreover $u\in C^2(\overline{B_R(y)})$ since $u_{{\bar R}}\in C^2(B_{\bar R}(y))$ and $R<{\bar R}$. Lastly let us observe that, from \eqref{5eq3}, we have 
\begin{equation*}
\begin{split}
 \left\|u\right\|_\infty&\leq\left\|u_{{\bar R}}\right\|_\infty\leq {\bar R}s_0({\bar R})={\bar R}\left(\frac{\beta}{{\bar R}pb}\right)^{\frac{1}{p-1}}\\
\left\|u'\right\|_\infty&=\left\|u'_{{\bar R}}\right\|_\infty= s_0(\bar R)=\left(\frac{\beta}{{\bar R}pb}\right)^{\frac{1}{p-1}}.
\end{split}
\end{equation*}	
This proves \eqref{unifestimate} in the case $R<\bar R$. The proof is then complete.
\end{proof}		

\begin{remark}\label{blowup}
{\rm
The radial solution $u$ of \eqref{extremeBallEq}, defined in \eqref{u0}, is also solution  of 
\begin{equation*}
			\left\{
				\begin{array}{cl}
					\beta\lambda_i(D^2u) + b|Du|^p = -M & \textrm{in } B_R(y)\\
					u = 0 & \textrm{on } \partial B_R(y)
				\end{array}
			\right.
		\end{equation*} 
		for any $i=2,\ldots,N$. This  is a trivial consequence of the inequality \eqref{radialEigenIneq}.\\
Moreover we point out that if $R=\bar R$, the function $u$ cannot be extended to a $C^2(\overline{B_R(y)})$ radial function, since \eqref{5eq1} and \eqref{5eq2} readily imply that $\displaystyle \lim_{r\to {R}^{^-}}u''(r)=-\infty$.
}
\end{remark}

In the next proposition we prove the existence of singular/blow-up solutions of the problem 
\begin{equation}\label{blow-upextremeBallEq}
			\left\{
				\begin{array}{cl}
					\beta\lambda_i(D^2u) + b|Du|^p = -M & \textrm{in } B_{\bar R}\backslash\left\{0\right\}\\
					u = 0 & \textrm{on } \partial B_{\bar R},
				\end{array}
			\right.
		\end{equation}
		where  $i=1,\ldots,N-1$. We shall follow the same argument of Proposition \ref{existOnBallProp} but now considering the \lq\lq second zero\rq\rq\ $s_2=s_2(r)$  of the function  $\varphi(r,s)=-\beta\frac sr+bs^p+M$.
		
		\begin{proposition}\label{blow-up solutions}
		There exists $u\in LSC\left([0,\bar R]\right)\cap C^2\left((0,\bar R)\right)$ radial solution of \eqref{blow-upextremeBallEq}. Moreover, if $p\in(1,2]$,  then $\displaystyle\lim_{r\to0^+}u(r)=+\infty$, while for $p>2$ the function $u$ is bounded.
		\end{proposition}
\begin{proof}
For $r\in(0,\bar R]$ let $s_1(r)=\left(\frac{\beta}{rpb}\right)^{\frac{1}{p-1}}$ be the  critical point of $\varphi(r,s)=-\beta\frac sr+bs^p+M$ and let 
\begin{equation}\label{3012eq4}
s_2(r)\geq s_1(r)
\end{equation} be such that $\varphi(r,s_2(r))=0$. It is clear that  
$$
\lim_{r\to0^+}s_2(r)=+\infty.
$$
Moreover, arguing similarly as in the proof of Proposition \ref{existOnBallProp}, we infer that
$$
s'_2(r)=-\frac{\beta s_2(r)}{r^2\partial_s\varphi(r,s_2(r))}<0,
$$
so the map $r\in(0,\bar R]\to s_2(r)$ is monotone decreasing, $s_2(\bar R)=s_1(\bar R)$ and $s_2\in C^1(0,\bar R)\cap C^0((0,\bar R])$. In this way the function defined by
\begin{equation}\label{3012eq3}
u(r)=\int_r^{\bar R}s_2(t)\,dt\qquad\text{for $r\in(0,\bar R]$},
\end{equation}
is a monotone decreasing  solution of \eqref{radialExtremeBallEq} and, differently from \eqref{u0}, it happens that 
\begin{equation}\label{3012eq1}
\lim_{r\to0^+}u'(r)=-\lim_{r\to0^+}s_2(r)=-\infty.
\end{equation}
In particular there is no way to extend smoothly $u(r)$ for $r=0$. Setting $\displaystyle u(0):=\lim_{r\to0^+}u(r)\in(0,\infty]$,   we obtain that 
$$u\in LSC([0,\bar R])\cap C^2\left((0,\bar R)\right).$$ 
In addition, since $u'$ is increasing, using \eqref{eq5} we infer that $\frac{u'}{r}$ is increasing too and then 
$$
\frac{u'(r)}{r}\leq u''(r)\qquad\text{for $r\in(0,\bar R)$}.
$$
Hence $u(|x|)$ is a classical solution of 
\begin{equation}\label{3012eq2}
			\left\{
				\begin{array}{cl}
					\beta\lambda_i(D^2u) + b|Du|^p = -M & \textrm{in } B_{\bar R}\backslash\left\{0\right\}\\
					u = 0 & \textrm{on } \partial B_{\bar R}
				\end{array}
			\right.
		\end{equation}
		for any $i=1,\ldots,N-1$. Moreover, in view of \eqref{3012eq1},  there are no test functions touching $u(|x|)$ from below at $x=0$. Then $u(|x|)$ is also a viscosity  supersolution of \eqref{3012eq2} in the whole ball $B_{\bar R}$.	 The function $u$ blows up in the origin for $p\leq 2$. Indeed, by \eqref{3012eq4}-\eqref{3012eq3}, we have
		$$
		\lim_{r\to0^+}u(r)=\int_0^{\bar R}s_2(t)\,dt\geq\int_0^{\bar R}s_1(t)\,dt=\left(\frac{\beta}{pb}\right)^{\frac{1}{p-1}}\int_0^{\bar R}\frac{dt}{t^{\frac{1}{p-1}}}=+\infty.
		$$
		If instead $p>2$ the function $u$ stays bounded. To show this let us first observe that for $r\in(0,\bar R)$ 
		$$
		\varphi\left(r, p^{\frac{1}{p-1}}s_1(r)\right)=M>0.
		$$
	This implies that $s_2(r)<p^{\frac{1}{p-1}}s_1(r)$ and 
		$$
		\lim_{r\to0^+}u(r)=\int_0^{\bar R}s_2(t)\,dt< p^{\frac{1}{p-1}}\int_0^{\bar R}s_1(t)\,dt=\left(\frac\beta b\right)^{\frac{1}{p-1}}\frac{p-1}{p-2}{\bar R}^{\frac{p-2}{p-1}}.
		$$
		Since $u$ is monotone decreasing and nonnegative, then $$\left\|u\right\|_\infty<\left(\frac\beta b\right)^{\frac{1}{p-1}}\frac{p-1}{p-2}{\bar R}^{\frac{p-2}{p-1}}$$ and the proof is complete.
		\end{proof}

		For the sake of completeness we  report the  explicit expressions of $s_0(r)$ and $s_2(r)$ in the model case: $p=\beta=2$ and  $b=M=\bar R=1$.  It holds that for any $r\in(0,1]$
		$$
		s_0(r)=\frac1r-\sqrt{\frac{1}{r^2}-1}\quad\text{and}\quad s_2(r)=\frac1r+\sqrt{\frac{1}{r^2}-1}.
		$$
		Recalling \eqref{u0} and \eqref{3012eq3},  by  straightforward computations we deduce that the functions
		\begin{equation*}
u_0(r)=\sqrt{1-r^2}-\log r-\frac12\log\frac{1+\sqrt{1-r^2}}{1-\sqrt{1-r^2}}
		\end{equation*}
		and 
		\begin{equation*}
u_2(r)=-\sqrt{1-r^2}-\log r+\frac12\log\frac{1+\sqrt{1-r^2}}{1-\sqrt{1-r^2}}
		\end{equation*}
		 are respectively radial solutions of 
		$$
		\left\{
				\begin{array}{cl}
					2\lambda_i(D^2u) + |Du|^2 = -1 & \textrm{in } B_{1}\\
					u = 0 & \textrm{on } \partial B_{1}
				\end{array}
			\right.
			\quad\text{and}\quad
			\left\{
				\begin{array}{cl}
					2\lambda_j(D^2u) + |Du|^2 = -1 & \textrm{in } B_{1}\backslash\left\{0\right\}\\
					u = 0 & \textrm{on } \partial B_{1}\\
					\displaystyle\lim_{x\to0}u(x)=+\infty
				\end{array}
			\right.
		$$
		for any $i=2,\ldots,N$ and $j=1,\ldots,N-1$.
		
		\begin{remark}
		{\rm Let us point that  the arguments of Propositions \ref{existOnBallProp}-\ref{blow-up solutions} also apply  if $M=0$, i.e. $\bar R=+\infty$ which means that the associated Dirichlet problem is solvable in any ball $B_R(y)$. In this case the zeros of $\varphi(r,s)=-\beta\frac{s}{r}+bs^p$ are explicit:
		$$
		s_0(r)=0\quad\text{and}\quad s_2(r)=\left(\frac{\beta}{br}\right)^{\frac{1}{p-1}}.
		$$
		So the function $u_0(r)=\int_r^Rs_0(t)\,dt$ is nothing more than the trivial solution of  
		\begin{equation*}
			\left\{
				\begin{array}{cl}
					\beta\lambda_N(D^2u) + b|Du|^p = 0& \textrm{in } B_R(y)\\
					u = 0 & \textrm{on } \partial B_R(y).
				\end{array}
			\right.
		\end{equation*}
		Instead, the radial function ($r=|x-y|$)
		$$
		u_2(r)=\int_r^Rs_2(t)\,dt=
		\left\{
		\begin{array}{cl}
		\left(\frac{\beta}{b}\right)^{\frac{1}{p-1}}\frac{p-1}{p-2}\left(R^\frac{p-2}{p-1}-r^\frac{p-2}{p-1}\right) & \text{if $p\neq2$}\\
		\frac{\beta}{b}\log\frac{R}{r} & \text{if $p=2$}
		\end{array}
		\right.
		$$
		is solution of 
		\begin{equation*}
			\left\{
				\begin{array}{cl}
					\beta\lambda_i(D^2u) + b|Du|^p =0& \textrm{in } B_R\backslash\left\{y\right\}\\
					u = 0 & \textrm{on } \partial B_R
				\end{array}
			\right.
		\end{equation*}
		for any $i=1,\dots,N-1$. 		Moreover for any $p\in(1,2]$ it turns out that $u_2$ is in fact solution of 
		\begin{equation*}
			\left\{
				\begin{array}{cl}
					\beta\lambda_i(D^2u) + b|Du|^p =0& \textrm{in } B_R\backslash\left\{y\right\}\\
					u = 0 & \textrm{on } \partial B_R\\
					\displaystyle\lim_{x\to y}u(x)=+\infty.
				\end{array}
			\right.
		\end{equation*}
		Note that, in the case of the Laplacian, the function $u(r)=c\left(r^\frac{p-2}{p-1}-R^\frac{p-2}{p-1}\right)$ is, for a suitable positive constant $c=c(N,p)$, solution of 
		\begin{equation}\label{lapl}
			\left\{
				\begin{array}{cl}
					\Delta u + |Du|^p =0& \textrm{in } B_R\backslash\left\{y\right\}\\
					u = 0 & \textrm{on } \partial B_R\\
					\displaystyle\lim_{x\to y}u(x)=+\infty
				\end{array}
			\right.
			\end{equation}
		only in the range $p\in\left(\frac{N}{N-1},2\right)$ and not for any $p\in(1,2)$. If $p=2$ the function $u(r)=(2-N)\log\frac rR$ is solution of \eqref{lapl} in dimension $N>2$.
		}
		\end{remark}

	
\subsection{Existence and uniqueness in uniformly convex domains}
	
For any continuous and bounded function $f$ in $\Omega$, let us define $\bar R=\bar R(\beta,b,c,p,f)$ by
\begin{equation}\label{RR}
{\bar R}=\frac{\beta(p-1)^{\frac{p-1}{p}}}{p\,b^\frac1p\,\left\|{(f-c)}^{^-}\right\|_\infty^{\frac{p-1}{p}}}.
\end{equation}
Note that $\bar R=+\infty$ if $\left\|{(f-c)}^{-}\right\|_\infty=0$, $\bar R<+\infty$ otherwise.\\
In the following Theorem \ref{exi} and Proposition \ref{CP} we shall assume that $\Omega\subseteq B_{{\bar R}}$ which, in the case $\bar R=+\infty$, simply means that $\Omega$ is a bounded domain.

The main result of this section is 
		\begin{theorem}[\textbf{Existence and uniqueness: superlinear case}]\label{exi}
		Let $\Omega\subset\RN$, $f\in C(\Omega)\cap L^\infty(\Omega)$. Suppose that $\Omega$ is a uniformly convex domain such that 
		\begin{equation}\label{suffcond}
		\Omega=\bigcap_{y\in Y}B_R(y)\qquad\text{for some}\qquad R\leq{\bar R}:=\frac{\beta(p-1)^{\frac{p-1}{p}}}{p\,b^\frac1p\,\left\|{(f-c)}^{^-}\right\|_\infty^{\frac{p-1}{p}}}\,.
		\end{equation}
		If \ref{ellipticity}-\ref{hom}, \ref{H1} and (CC) hold then there exists a unique   $u\in C(\overline\Omega)$ viscosity solution of \eqref{genEq}.   
		\end{theorem}
In order to prove Theorem \ref{exi}, we shall need a comparison principle that is given below, and a proof that relies on the results obtained in the radial setting. But before proceeding with the proof we shall make the following
\begin{remark}\label{por}\rm{ Let us recall that
 Porretta in \cite{PorLin} proved the following dichotomy:  in the case $1<p\leq 2$,  if there exists a solution $u$ of 
\begin{equation}\label{delt0}
			\left\{
				\begin{array}{cl}
					\Delta u + |Du|^p = f(x) & \textrm{in }\Omega\\
					u = 0 & \textrm{on }\partial\Omega,
				\end{array}
			\right.
		\end{equation}
then, for $\lambda>0$, the solutions $u_\lambda$ of
\begin{equation}\label{deltlam}
			\left\{
				\begin{array}{cl}
					\Delta u + |Du|^p -\lambda u= f(x) & \textrm{in }\Omega\\
					u = 0 & \textrm{on }\partial\Omega
				\end{array}
			\right.
		\end{equation}	
converge to  $u$ as $\lambda\rightarrow 0$.
If, instead there are no solutions of \eqref{delt0}, then, for $\lambda\rightarrow 0$ , $u_\lambda\rightarrow+\infty$ and $-\lambda u_\lambda$ converges to the ergodic constant $c_o$, i.e. the unique value such that the following problem
\begin{equation}\label{delterg}
			\left\{
				\begin{array}{cl}
					\Delta v + |Dv|^p +c_o= f(x) & \textrm{in }\Omega\\
					v = -\infty & \textrm{on }\partial\Omega
				\end{array}
			\right.
		\end{equation}	
has  a solution. The existence of $c_o$ goes back to the acclaimed work of Lasry and 
Lions \cite{LL}.
Analogous results have been obtained for nonlinear operators both in divergence form and fully nonlinear, see \cite{LP} and \cite{BDL}.
The relation of $c_o$ to the large-time behavior of solutions of the associated evolution problem can also be used to determine the existence or nonexistence of bounded solutions of \eqref{delt0} in the case $p>2$, see \cite{tabet2010large}.
The question of whether the ergodic constant exists and if this dichotomy is also valid in our generality is far from obvious since any compactness result is very hard to 
establish for such degenerate operators.}
\end{remark}
We now proceed with the comparison principle.
\begin{proposition}[\textbf{Comparison principle: superlinear case}]\label{CP}
	Let $\Omega$ be a bounded domain and let $f\in C(\Omega)\cap L^\infty(\Omega)$ be such that $\Omega\subseteq B_{{\bar R}}(y)$ for some $y\in\RN$. Suppose that \ref{ellipticity}-\ref{hom}, \ref{H1} and (CC) hold.  
	 If $u\in USC(\overline{\Omega})$, $v\in LSC(\overline{\Omega})$ are respectively  sub and supersolution of \eqref{genEq}, then $u\leq v$ in $\overline{\Omega}$.
	\end{proposition}
		\begin{remark}{\rm
		We stress that in Proposition \ref{CP}  we do not require convexity assumptions on $\Omega$.}
		\end{remark}
			
		\begin{proof}[Proof of Proposition \ref{CP}]
		We assume by contradiction that there exists $z_0\in\Omega$ such that $u(z_0)>v(z_0)$. Let us first treat the case ${\bar R}<+\infty$.
		Fix $\delta>0$ such that $\delta<{\rm dist}(z_0,\partial B_{{\bar R}}(y))$. In this way $z_0\in B_{{\bar R}-\delta}(y)$ and for any $\varepsilon>0$ sufficiently small we have that 
		\begin{equation}\label{2eq1}
		\bar R-\delta< \frac{\beta(p-1)^{\frac{p-1}{p}}}{p\,b^\frac1p\,\left(\left\|{(f-c)}^{^-}\right\|_\infty+\varepsilon\right)^{\frac{p-1}{p}}}\,.
		\end{equation}
		By \eqref{2eq1} and using Proposition \ref{existOnBallProp}, there exists $\varphi\in C^2(\overline{B_{\bar R-\delta}(y)})$ solution of 
		\begin{equation}\label{3eq2}
		\beta\lambda_N(D^2\varphi)+b|D\varphi|^p=-\left(\left\|{(f-c)}^{-}\right\|_\infty+\varepsilon\right)\qquad\text{in $B_{\bar R-\delta}(y)$}.
		\end{equation}
		For any $\sigma\in(0,1)$ let us consider the convex combination 
		$$v_\sigma(x)=\sigma v(x)+(1-\sigma)\varphi(x),\qquad x\in\Omega\cap B_{{\bar R}-\delta}(y).$$
		We claim that $v_\sigma$ satisfies, in the viscosity sense, the inequality
		\begin{equation}\label{3eq3}
		F(x,D^2v_\sigma)+H(Dv_\sigma)\leq f(x)-(1-\sigma)\varepsilon\qquad\text{in $\Omega\cap B_{\bar R-\delta}(y)$}.
		\end{equation}
		For this let $x_0\in\Omega\cap B_{\bar R-\delta}(y)$ and let $\psi\in C^2(\Omega\cap B_{\bar R-\delta}(y))$ such that 
		$v_\sigma -\psi$ has a local minimum at $x_0$. Hence  $v-\frac1\sigma(\psi-(1-\sigma)\varphi)$ has a local minimum at $x_0$. Since $v$ is supersolution of \eqref{genEq}, then setting
		$$
		\eta=\frac1\sigma D(\psi-(1-\sigma)\varphi)(x_0),\quad X=\frac1\sigma D^2(\psi-(1-\sigma)\varphi)(x_0)
		$$
		we have 
		$$
		F\left(x_0,X\right)+H\left(\eta\right)\leq f(x_0).
		$$
		Then, using the assumptions \ref{ellipticity}-\ref{hom}, the equation \eqref{3eq2} and the fact that $D^2\varphi(x)\leq0$ for any $x\in B_{\bar R-\delta}(y)$, we have
		\begin{equation*}
		\begin{split}
		F(x_0,D^2\psi(x_0))+H(D\psi(x_0))&=\sigma F\left(x_0,X+\frac{1-\sigma}{\sigma}D^2\varphi(x_0)\right)+\sigma H\left(\eta\right)\\
		&\quad+H(\sigma\eta+(1-\sigma)D\varphi(x_0))-\sigma H\left(\eta\right)\\
		&\leq \sigma \left(F\left(x_0,X\right)+H(\eta)\right)\\
		&\quad+(1-\sigma)\left(\beta\lambda_N(D^2\varphi(x_0))+b|D\varphi(x_0)|^p+c\right)\\
		&\leq\sigma f(x_0)+(1-\sigma)\left(c-(\left\|{(f-c)}^{-}\right\|_\infty+\varepsilon)\right)\\
		&\leq f(x_0)-(1-\sigma)\varepsilon.
		\end{split}
		\end{equation*}
		This proves the claim. Now, since by \eqref{3eq3} the function $v_\sigma$ satisfies a strict inequality, by the assumption (CC) we infer that
		$$
		u(z_0)-v_\sigma(z_0)\leq\sup_{\partial\left(\Omega\cap B_{\bar R-\delta}(y)\right)}(u-v_\sigma).
		$$
		Passing to the limit as $\sigma\to1^-$ we have 
		$$
		u(z_0)-v(z_0)\leq\sup_{\partial\left(\Omega\cap B_{\bar R-\delta}(y)\right)}(u-v).
		$$
		The above inequality holds for any positive $\delta<{\rm dist}(z_0,\partial B_{\bar R}(y))$. Thus 
		$$
		u(z_0)-v(z_0)\leq\sup_{\partial\Omega}(u-v)\leq0,
		$$
		contradiction.
		
		The proof in the case $\bar R=+\infty$, i.e. $f(x)\geq c$ for any $x\in\Omega$, is in fact a little bit easier since we don't have to consider the parameter $\delta$.	Take $R>0$  such that $\overline\Omega\subset B_R$. For $\varepsilon>0$ sufficiently small,  the condition \eqref{condR} is then satisfied with $M=\varepsilon$. Hence there exists a radial solution of 
		\begin{equation*}\label{2eq2}
		\beta\lambda_N(D^2\varphi)+b|D\varphi|^p=-\varepsilon\qquad\text{in $B_R$}
		\end{equation*}
		which is concave in $B_R$.
		Then, as above, $$
		u(z_0)-v_\sigma(z_0)\leq\sup_{\partial\Omega}(u-v_\sigma)
		$$
		and again we obtain a contradiction sending $\sigma\to1^-$.
		\end{proof}
		
		\begin{remark}\label{CPMASuper}
		\rm In the case of the Monge-Ampère equation
		\begin{equation}\label{44eq1}
		\left(\det(D^2u)\right)^\frac1N = f(x)+b|Du|^p \quad \text{in $\Omega$},
		\end{equation}
		with $f\geq0$ and $p\geq1$, the validity of the comparison principle is a consequence of  \cite[Theorem V.2]{IL},  which in fact applies to more general right hand side $f(x,u,Du)$, locally Lipschitz in the gradient variable.\\ 
		On the other hand the proof of Proposition \ref{CP}, in particular the perturbation argument to produce a strict supersolution, can be easily adapted to the equation \eqref{44eq1}, in order to obtain comparison principle within the class of  convex sub and supersolutions. For the sake of completeness we briefly describe this aspect.
				Let $u$ and $v$ be respectively  subsolution and supersolution of \eqref{44eq1} such that $u\leq v$ on $\partial\Omega$ and $u(z_0)>v(z_0)$ for some $z_0\in\Omega$. First assume $\bar R<\infty$. Fix $y\in Y$,  $\delta>0$ such that $z_0\in B_{\bar R-\delta}(y)$. For any $\varepsilon>0$ such that 
					\begin{equation*}
		\bar R-\delta< \frac{\beta(p-1)^{\frac{p-1}{p}}}{p\,b^\frac1p\,\left(\left\|f\right\|_\infty+\varepsilon\right)^{\frac{p-1}{p}}}
		\end{equation*}
		there exists, by Proposition \ref{existOnBallProp}, a convex solution $\varphi\in C^2(\overline{B_{\bar R-\delta}(y)})$		of 
			$$
			\lambda_1(D^2\varphi)=\left\|f\right\|_\infty+\varepsilon+b|D\varphi|^p\quad\text{in $B_{\bar R-\delta}(y)$}.
			$$
			Hence by \eqref{deg2} we infer that the function $u_\sigma(x)=\sigma u(x)+(1-\sigma)\varphi(x)$, where $\sigma\in(0,1)$ and $x\in\Omega\cap B_{\bar R-\delta}(y)$, satisfies in the viscosity sense the inequalities
			\begin{equation*}
			\begin{split}
			\left(\det(D^2u_\sigma(x))\right)^\frac1N&\geq\sigma\left(\det(D^2u(x))\right)^\frac1N+(1-\sigma)\lambda_1(D^2\varphi)\\
			&\geq\sigma\left(f(x)+b|Du(x)|^p\right)+(1-\sigma)\left(\left\|f\right\|_\infty+\varepsilon+b|D\varphi(x)|^p\right)\\
			&\geq b|Du_\sigma(x)|^p+f(x)+(1-\sigma)\varepsilon.
			\end{split}
			\end{equation*}
	Since $u_\sigma$ satisfies a strict inequality, we can use \cite[Chapters 3 and 5.C]{CIL} (note  that the matrices $X,Y\in{\mathbb S}^N$  in (3.9)-(3.10) of \cite{CIL} satisfy $0\leq X\leq Y$ since $u$ is convex) to reach a contradiction,  letting first $\sigma\to1^-$ and then $\delta\to0^+$.\\
	If $\bar R=+\infty$, i.e. $f\equiv0$, then we consider $R>0$ such that $\overline\Omega\subset B_R$. For $\varepsilon$ small enough there exists, by Proposition \ref{existOnBallProp}, a convex solution $\varphi\in C^2(\overline{B_R})$ 
	of $$
			\lambda_1(D^2\varphi)=\varepsilon+b|D\varphi|^p\quad\text{in $B_R$}
			$$
			and as above we get a contradiction in the limit $\sigma\to1^-$.  
		\end{remark}
		
We can start the 		
		\begin{proof}[Proof of Theorem \ref{exi}]
		By Proposition \ref{CP} we are in position to use the Perron's method. By \eqref{maxinequ} it is  sufficient to construct a concave supersolution  $\overline{u}\in C(\overline\Omega)$ of
		\begin{equation}\label{eqsuper}
		\beta\lambda_N(D^2u)+b|Du|^p= f(x)-c\qquad\text{in $\Omega$}
		\end{equation}		
and convex subsolution $\underline{u}\in C(\overline\Omega)$ of 
\begin{equation}\label{eqsub}
	\beta\lambda_1(D^2u)= f(x)+d\qquad\text{in $\Omega$}.
		\end{equation}
		such that $\overline{u}=\underline{u}=0$ on $\partial\Omega$.
		
			We first consider the supersolution case. If $\left\|(f-c)^-\right\|_\infty=0$ just take $\overline{u}=0$. Suppose that $\left\|(f-c)^-\right\|_\infty\neq0$. By proposition \ref{existOnBallProp}, with $M=\left\|(f-c)^-\right\|_\infty$, for any $y\in Y$ there exists $v_y(x)=v(|x-y|)$ a concave solution of 
		\begin{equation}\label{eqsuper1}
		\left\{
				\begin{array}{cl}
					\beta\lambda_N(D^2v) + b|Dv|^p \leq f(x)-c & \textrm{in } B_R(y)\\
					v = 0 & \textrm{on } \partial B_R(y).
				\end{array}
			\right.
			\end{equation}
	In view of the uniform estimates \eqref{unifestimate}, the family $\left\{v_y\right\}_{y\in Y}$ is bounded in $C^1(\overline\Omega)$. Hence we infer that  $$\overline{u}(x):=\inf_{y\in Y}v_y(x)$$  is a nonnegative Lipschitz continuous function in $\overline\Omega$ which is a concave supersolution of \eqref{eqsuper}.	Moreover for any $x\in\partial\Omega$, there exists $y_x\in Y$  such that $x\in\partial B_R(y_x)$. Thus $\overline{u}(x)\leq v_{y_x}(x)=0$, leading to the Dirichlet condition $\overline{u}=0$ on $\partial \Omega$. In addition, since $\overline{u}$ is a concave supersolution of \eqref{eqsuper} and $\left\|(f-c)^-\right\|_\infty\neq0$, we infer that $\overline{u}$ is in fact positive in $\Omega$. 
			
	The construction of a subsolution to \eqref{eqsub} is more direct than the supersolution since there are no gradient terms in the equation \eqref{eqsub}. For instance
	the function
	$$
	u_y(x)=\frac{\left\|(f+d)^+\right\|_\infty}{2\beta}\left(|x-y|^2-R^2\right)
	$$
	is, for any $y\in Y$, a negative convex solution of
	\begin{equation*}
		\left\{
				\begin{array}{cl}
					\beta\lambda_1(D^2u) \geq f(x)+d & \textrm{in } B_R(y)\\
					u = 0 & \textrm{on } \partial B_R(y).
				\end{array}
			\right.
			\end{equation*}
	 Since $\left\{u_y\right\}_{y\in Y}$ is bounded in $C^1(\overline\Omega)$, then  $$\underline{u}(x):=\sup_{y\in Y}u_y(x)$$ is in turn a  convex subsolution of \eqref{eqsub}, continuous in $\overline\Omega$  such that $\underline{u}=0$ on $\partial\Omega$.		
		\end{proof}

		We end this section by showing that the above arguments can be slightly modified in order to obtain the existence and uniqueness of convex solutions of 
		
		\begin{equation}\label{pbmongeampere}
			\left\{
				\begin{array}{cl}
					\left(\det(D^2u)\right)^\frac1N = f(x)+b|Du|^p & \textrm{in }\Omega\\
					u = 0 & \textrm{on }\partial\Omega.
				\end{array}
			\right.
		\end{equation}
		Since  equations of the form $$\det (D^2u)=g(x,Du)\quad\text{in $\Omega$}$$ are elliptic only when $D^2u$ is nonnegative,  it is therefore  natural to confine the attention to convex solutions $u$ and positive right hand side $g$. In view of this we shall assume that $f\geq0$ in $\Omega$ and that, differently from  \eqref{mainexample},  the gradient term is placed on the right hand side of the equation with a positive sign in front of it.
		
		\begin{theorem}\label{MongeA1}
		Let $\Omega\subset\R^N$, $f\in C(\Omega)\cap L^\infty(\Omega)$ be nonnegative in $\Omega$. If $\Omega$ is a uniformly convex domain such that 
		\begin{equation}\label{suffcond2}
		\Omega=\bigcap_{y\in Y}B_R(y)\qquad\text{for some}\qquad R\leq{\bar R}:=\frac{(p-1)^{\frac{p-1}{p}}}{p\,b^\frac1p\,\left\|f\right\|_\infty^{\frac{p-1}{p}}}\,,
		\end{equation}
		then there exists a unique $u\in C(\overline\Omega)$ viscosity solution of \eqref{pbmongeampere}.
		\end{theorem}

		\begin{proof}
		The function $\overline u=0$ is a trivial supersolution of \eqref{pbmongeampere}. By Theorem \ref{exi}, with $c=0$ and $f$
		replaced by $-f$, 		there exists $u$  concave solution of  
		\begin{equation*}
		\left\{
				\begin{array}{cl}
					\lambda_N(D^2u)+b|Du|^p=-f(x) & \textrm{in } \Omega\\
					u = 0 & \textrm{on } \partial \Omega.
				\end{array}
			\right.
			\end{equation*}
		Hence $\underline u=-u$ is a convex function satisfying 
				\begin{equation*}
		\left\{
				\begin{array}{cl}
					\lambda_1(D^2\underline u)=f(x)+b|D\underline u|^p & \textrm{in } \Omega\\
					\underline u = 0 & \textrm{on } \partial \Omega.
				\end{array}
			\right.
			\end{equation*}
				Since $\det(X)\geq\left(\lambda_1(X)\right)^N$ for any $X\geq0$, the function $\underline u$  is in turn subsolution of \eqref{pbmongeampere}. 
				
In view of Remark \ref{CPMASuper}, which ensures the validity of the comparison principle, we are in position to use the Perron's method. Since we are interested in convex solutions it is necessary to check that every manipulations of sub/supersolutions involved in the application of this method, continues to produce  convex functions. In particular we have to look at  the \lq\lq bump\rq\rq construction. 
		Let 
	\begin{equation}\label{PerronMA}
		u(x):=\sup\left\{w(x)\,:\,\underline u\leq w\leq0\;\text{and $w$ is a convex subsolution of \eqref{pbmongeampere}}\right\}.
		\end{equation}
	 If $u(x)$ fails to be a supersolution at $\hat x\in\Omega$ (note that $u$ is continuous in $\Omega$ since it is convex), then there would be a test function $\varphi\in C^2(\Omega)$ such that $u(\hat x)=\varphi(\hat x)$, $u(x)\geq\varphi(x)$ in $\Omega$ and 
		$$
		\left(\det(D^2\varphi(\hat x))\right)^\frac1N>f(\hat x)+b|D\varphi (\hat x)|^p.
		$$
		By continuity, for $\delta>0$ small enough, the function
		$\psi(x)=\varphi(x)+\frac{\delta^3}{8}-\frac{\delta}{2}|x-\hat x|^2$
		is convex in $B_\delta(\hat x)$ and moreover it  satisfies, in the classical sense, the inequality
		$$
		\left(\det(D^2\psi(x))\right)^\frac1N>f(x)+b|D\psi (x)|^p \quad \text{in $B_\delta(\hat x)$}.
		$$
			For $x\in\Omega$ such that $\frac\delta2\leq|x-\hat x|\leq \delta$ it turns out that $\psi(x)\leq u(x)$ and that the function
			$$
			U(x)=\left\{
			\begin{array}{rl}
			\max\left\{u(x),\psi(x)\right\} & \text{if $x\in B_\delta(\hat x)$}\\
			u(x) & \text{otherwise}
			\end{array}
			\right.
			$$
			is a convex viscosity subsolution of \eqref{pbmongeampere} and by construction $U\geq\underline u$. Moreover  $U\leq0$ on $\partial \Omega$. By comparison $U\leq0$ in $\Omega$. Hence $U(x)\leq u(x)$ in $\Omega$ by the definition \eqref{PerronMA} of $u$, but  $U(\hat x)=\psi(\hat x)=u(\hat x)+\frac{\delta^3}{3}$ therefore contradicting the maximality of $u$.
		This shows that \eqref{PerronMA} is in fact the unique convex solution of \eqref{pbmongeampere}.
		\end{proof}

		\section{Sublinear Hamiltonians}\label{sub}
		Throughout this section we shall assume $p<1$. Let us first note that  the weak comparison principle may fail for  
$$
F(x,D^2u)+|Du|^p=f(x)\quad\text{in $\Omega$}
$$
without extra condition on $f$.
 We list some  examples of nontrivial solutions (in the ball $B_R$ with zero Dirichlet boundary condition) for some equations of interest in this paper:
\begin{equation*}
\begin{array}{ll}

\lambda_1(D^2u)+|Du|^p=0 &\quad \displaystyle u(x)=\frac{(1-p)^{\frac{2-p}{1-p}}}{2-p}\left(R^{\frac{2-p}{1-p}}-|x|^{\frac{2-p}{1-p}}\right)\\
\, &\, \\
\lambda_i(D^2u)+|Du|^p=0\;\;\text{for $i=2,\ldots,N$} &\quad  \displaystyle u(x)=\frac{(1-p)}{2-p}\left(R^{\frac{2-p}{1-p}}-|x|^{\frac{2-p}{1-p}}\right)\\
\, &\, \\
\Delta u+|Du|^p=0 &\quad  \displaystyle u(x)=\frac{(1-p)}{(2-p)(N-1+\frac{1}{1-p})^{\frac{1}{1-p}}}\left(R^{\frac{2-p}{1-p}}-|x|^{\frac{2-p}{1-p}}\right)\\
\, &\, \\
\left(\det(D^2u)\right)^\frac{1}{N}=|Du|^p &\quad  \displaystyle u(x)=\frac{(1-p)^{\frac{1+N(1-p)}{N(1-p)}}}{2-p}\left(|x|^{\frac{2-p}{1-p}}-R^{\frac{2-p}{1-p}}\right)\,.
\end{array}
\end{equation*}

To obtain the existence of a unique viscosity solution we impose a sign condition on the forcing term $f(x)$. Consider the problem
\begin{equation}\label{genEq2}
			\left\{
				\begin{array}{cl}
					F(x,D^2u) + H(Du) = f(x) & \textrm{in }\Omega\\
					u = 0 & \textrm{on }\partial\Omega
				\end{array}
			\right.
		\end{equation}
		and assume: 
		\begin{enumerate}[label=(S\arabic*)]
	 \item\label{S1} $f\in C(\Omega)\cap L^\infty(\Omega)$ and $\displaystyle \sup_{\Omega}f<0$.
	\end{enumerate}

	\begin{theorem}[\textbf{Existence and uniqueness: sublinear case}]\label{exisublinear}
	Let $\Omega\subset\RN$ be a bounded uniformly convex domain. Under the assumptions \ref{ellipticity}-\ref{hom}, \ref{H2}, \ref{S1} and (CC) there exists a unique $u\in C(\overline\Omega)$ viscosity solution of \eqref{genEq2}.
	\end{theorem}
	
	\begin{remark}
	{\rm
	Differently from the case $p>1$, see \eqref{suffcond},  in the sublinear setting we prove existence and uniqueness of solutions for any $R$.
	}
	\end{remark}
	
	\begin{proposition}\label{exiballsublinear}
		For any $M>0$ and any ball $B_R(y)$ there exists  $u\in C^2(\overline{B_R(y)})$ radial solution of 
		\begin{equation}\label{extremeBallEq3}
			\left\{
				\begin{array}{cl}
					\beta\lambda_N(D^2u) + b|Du|^p = -M & \textrm{in } B_R(y)\\
					u = 0 & \textrm{on } \partial B_R(y).
				\end{array}
			\right.
		\end{equation}
		Moreover
		\begin{equation}\label{C2estimate}
		\left\|u\right\|_{C^1(\overline{B_R(y)})}\leq (1+R)\max\left\{M^\frac1p,\left(\frac{(1+b)R}{\beta}\right)^{\frac{1}{1-p}}\right\}.
		\end{equation}
		
	\end{proposition}	
	\begin{proof}
	For  $R>0$ let us consider the ODE problem 
	\begin{equation}\label{31eq1}
					\left\{
						\begin{array}{cl}
							\beta\frac{u'(r)}{r} + b(-u'(r))^p = -M, & r\in (0,R)\\
							u''(r)\leq \frac{u'(r)}{r}<0, & r\in (0,R)\\
							u'(0)=u(R)=0.
						\end{array}
					\right.
				\end{equation}
				We claim that such problem has a unique solution. Similarly to the proof of Proposition \ref{existOnBallProp}, let us observe that $u$ is solution of the equation in \eqref{31eq1} if, and only if, for any $r\in(0,R)$ it holds that $\varphi(r,-u'(r))=0$,  where 
				$$
\varphi(r,s)=-\beta\frac sr+bs^p+M\qquad \text{for $r,s>0$.}
$$
Now, since $p\in(0,1)$, differently from the proof of Proposition \ref{existOnBallProp}, for any fixed $r\in(0,+\infty)$, the function $s\mapsto\varphi(r,s)$ is concave in $[0,\infty)$. Moreover it is increasing for $s\in[0,s_1(r)]$, then decreasing for $s\geq s_1(r)$, where
\begin{equation}\label{31eq2}
s_1(r)=\left(\frac{bpr}{\beta}\right)^{\frac{1}{1-p}}.
\end{equation} 
Moreover $\displaystyle\lim_{s\to+\infty}\varphi(r,s)=-\infty$. Hence there exists (a unique) $s_0=s_0(r)> s_1(r)$ such that $\varphi(r,s_0(r))=0$. Since  $$\partial_s\varphi(r,s_0(r))=-\frac\beta r+bp(s_0(r))^{p-1}<-\frac\beta r+bp(s_1(r))^{p-1}=0\qquad \text{ for $r>0$},$$ by the implicit function theorem, $s_0(r)\in C^1(0,+\infty)$ and $s_0'(r)>0$ for $r>0$. In fact $s_0\in C^1([0,+\infty))$. To show this we first observe that, as in  (iii)-Proposition \ref{existOnBallProp} we have $\displaystyle\lim_{r\to0^+}s_0(r)=0$, since $\varphi(r,r^\alpha)<0$ for  $\alpha\in(0,1) $ and $r$ sufficiently small. Moreover, exploiting the definition of $s_0$, i.e. $\varphi(r,s_0(r))=0$, we have
\begin{equation}\label{31eq3}
s'_0(r)=-\frac{\beta s_0(r)}{r^2\partial_s\varphi(r,s_0(r))}=\frac{M+b(s_0(r))^p}{\beta-bp\frac{r}{(s_0(r))^{1-p}}}.
\end{equation}
Since 
$$
\varphi\left(r,\frac{Mr}{\beta}\right)=b\left(\frac{Mr}{\beta}\right)^p>0,
$$
we deduce that $s_0(r)>\frac{Mr}{\beta}$ for $r>0$. Then $$0<\frac{r}{(s_0(r))^{1-p}}<\left(\frac{\beta}{M}\right)^{1-p}r^p\qquad\text{for $r>0$}$$ and 
$$
\lim_{r\to0^+}\frac{r}{(s_0(r))^{1-p}}=0.
$$ 
 Form this we can pass to the limit, as $r\to0^+$, in \eqref{31eq3} and obtain that  
$$ 
\lim_{r\to0^+}s_0'(r)=\frac M\beta.
$$
This shows that $s_0(r)\in C^1([0,\infty))$.\\
Summing up the function $s_0=s_0(r)\in C^1([0,R])$ for any $R>0$ and satisfies
\begin{equation}\label{31eq4}
	s_0(0)=0\;,\qquad s'_0(r)>0 \quad\text{for $r\in[0,R]$}.
	\end{equation}
Moreover, since
$$
\beta \frac{s_0(r)}{r}=b(s_0(r))^p+M\qquad\text{for $r>0$},
$$ 
we infer that the map $r\mapsto\frac{s_0(r)}{r}$ is increasing. In particular
\begin{equation}\label{31eq5}
s'_0(r)\geq\frac{s_0(r)}{r}\qquad\text{for $r\in(0,R]$.}
\end{equation}
	
For $r\in[0,R]$ let $u=u(r)$ be the function defined by the formula
$$
u(r):=\int_r^Rs_0(t)\,dt.
$$
By \eqref{31eq4}-\eqref{31eq5} it follows that $u$ is the solution of \eqref{31eq1} and, slightly abusing notation, that $u(x) = u(r)$ with $r=|x-y|$ is solution of \eqref{extremeBallEq3}. To complete the proof it remains to prove the estimate \eqref{C2estimate}. Since $u$ is monotone decreasing and concave, then
\begin{equation}\label{61eq1}
\left\|u\right\|_{C^1(\overline{B_R(y)})}=u(0)-u'(R)\leq(1+R)s_0(R).
\end{equation}
Moreover, setting $$\bar s=\max\left\{M^\frac1p,\left(\frac{(1+b)R}{\beta}\right)^{\frac{1}{1-p}}\right\},$$
by a straightforward computation we have 
$$
\varphi(R,\bar s)=-\beta\frac{\bar s}{r}+b{\bar s^p}+M\leq0.
$$
Hence $s_0(R)\leq\bar s$. From \eqref{61eq1} we obtain \eqref{C2estimate}.
	\end{proof}
	
	\begin{proposition}[\textbf{Comparison principle: sublinear case}]\label{CP2}
	Let $\Omega$ be a bounded domain. Suppose that \ref{hom}, \ref{H2}, \ref{S1} and (CC) hold.  
	 If $u\in USC(\overline{\Omega})$, $v\in LSC(\overline{\Omega})$ are respectively  sub and supersolution of \eqref{genEq2}, then $u\leq v$ in $\overline{\Omega}$.
	\end{proposition}
	\begin{proof}
	For $\varepsilon>0$ the function $v_\varepsilon(x)=(1+\varepsilon) v(x)\in LSC(\overline{\Omega})$ is a strict supersolution of \eqref{genEq2}, since by \ref{hom}, \ref{H2} and \ref{S1} it holds that the following inequalities hold in the viscosity sense:
	$$
	F(x,D^2v_\varepsilon)+H(Dv_\varepsilon)\leq-(1+\varepsilon)H(Dv)+(1+\varepsilon)f(x)+H(Dv_\varepsilon)\leq f(x)+\varepsilon\sup_{\Omega}f\quad\text{in $\Omega$}.
	$$
	By (CC) we have $u\leq v_\varepsilon$ in $\Omega$ and we conclude by sending $\varepsilon\to0^+$.	\end{proof}
	
	\begin{proof}[Proof of Theorem \ref{exisublinear}]
	In view of Proposition \ref{exiballsublinear}, with $M=\left\|f\right\|_\infty+c$, for any ball $B_R(y)$ there exists a concave solution $u_y=u_y(x)$ of 
	\begin{equation*}
			\left\{
				\begin{array}{cl}
					\beta\lambda_N(D^2u) + b|Du|^p = -(\left\|f\right\|_\infty+c) & \textrm{in } B_R(y)\\
					u = 0 & \textrm{on } \partial B_R(y).
				\end{array}
			\right.
		\end{equation*}
	Since $\displaystyle\Omega=\bigcap_{y\in Y}B_R(y)$, then by the structural conditions \ref{ellipticity}, \ref{H2}, we infer that for any $y\in Y$ the function $u_y$ is in turn a supersolution of \eqref{genEq2}. Using the stability of the supersolution property under inf-operation and the uniform (with respect to $y$) estimate \eqref{C2estimate}, we infer that $\displaystyle v(x):=\inf_{y\in Y}u_y(x)\in\rm{Lip}(\overline\Omega)$ is a positive concave supersolution of \eqref{genEq2} such that $v=0$ on $\partial\Omega$. On the other hand, by the assumption \ref{S1}, the function $u=0$ is a subsolution of \eqref{genEq2}. By the Perron's method and by Proposition \ref{CP2} we conclude that $$
	u(x):=\sup\left\{w(x)\,:\,0\leq w\leq v\;\;\text{and $w$ is a subsolution of \eqref{genEq2}}\right\}
	$$
	is the unique viscosity solution of \eqref{genEq2}.
	\end{proof}
	
	We end this section by adapting the above argument to get  existence and uniqueness of convex solutions for the sublinear, $p<1$, Monge-Ampère equation
	\begin{equation}\label{pbmongeampere2}
			\left\{
				\begin{array}{cl}
					\left(\det(D^2u)\right)^\frac1N = f(x)+b|Du|^p & \textrm{in }\Omega\\
					u = 0 & \textrm{on }\partial\Omega.
				\end{array}
			\right.
		\end{equation}
		Assumption \ref{S1} is here replaced by
		\begin{enumerate}[label=(S'\arabic*)]
	 \item\label{S2} $f\in C(\Omega)\cap L^\infty(\Omega)$ and $\displaystyle \inf_{\Omega}f>0$.
	\end{enumerate}
	
	\begin{remark}\label{CPMA}
		\rm It is worth to point out that  the standard Lipschitz continuity assumption in the gradient variable, see \cite[(5.22)-Theorem V.2]{IL},  is  no longer true in our framework since $p$ is strictly less than 1. Nevertheless  arguing in a similar way to the proof of Proposition \ref{CP2},  using the assumption \ref{S2} and the fact that $p<1$, it is readily seen that  the comparison principle between sub and supersolutions of \eqref{pbmongeampere2}  still holds.
		\end{remark}
	
	\begin{theorem}
	Let $\Omega\subset\RN$ be a bounded uniformly convex domain. If \ref{S2} hold, then there exists a unique $u\in C(\overline\Omega)$ viscosity solution of \eqref{pbmongeampere2}.
	\end{theorem}
		\begin{proof}
		Theorem \ref{exisublinear} yields the existence of a viscosity solution $\underline u$ of 
		\begin{equation*}
		\left\{
				\begin{array}{cl}
					\lambda_1(D^2\underline u)=f(x)+b|D\underline u|^p & \textrm{in } \Omega\\
					\underline u = 0 & \textrm{on } \partial \Omega.
				\end{array}
			\right.
			\end{equation*}
			From this and using the inequality $\det(X)\geq\left(\lambda_1(X)\right)^N$ for any $X\geq0$, we infer that $\underline u$ is a subsolution of \eqref{pbmongeampere2}. The function $\overline u=0$ is instead a trivial supersolution of \eqref{pbmongeampere2}.
			In view of Remark \ref{CPMA}, we can use the Perron's method in the class of convex subsolution (with the same observations as those did  in the superlinear case, see the proof of Theorem \ref{MongeA1}), so obtaining  the existence  of a unique convex viscosity solution  of \eqref{pbmongeampere2}. 
		\end{proof}

	\bigskip
\noindent
\textsc{I. Birindelli, G. Galise}: Dipartimento di Matematica \lq\lq Guido Castelnuovo\rq\rq,\\
Sapienza Universit\`a di Roma, P.le Aldo Moro 2, I-00185 Roma, Italy.\\
E-mail: \texttt{isabeau@mat.uniroma1.it}\\
E-mail: \texttt{galise@mat.uniroma1.it}

\smallskip
\noindent\textsc{A. Rodr\'iguez}: Departamento de Matem\'atica y Ciencia de la Computaci\'on,\\
Universidad de Santiago de Chile, Avda. Libertador General Bernardo O’Higgins 3383, Santiago, Chile.\\
E-mail: \texttt{andrei.rodriguez@usach.cl}\\
	
\end{document}